\documentclass[12pt,a4paper,reqno]{amsart}
\usepackage[all]{xy}
\usepackage{amsmath}
\allowdisplaybreaks
\usepackage{amsfonts}
\usepackage{amssymb,amsthm,amsfonts,amsthm,latexsym,enumerate,url,cases}
\numberwithin{equation}{section}
\usepackage{mathrsfs}
\usepackage{hyperref}
\usepackage{extarrows}
\hypersetup{colorlinks=true,citecolor=blue,linkcolor=blue,urlcolor=blue}
\newtheorem{theorem*}{Theorem}
\newtheorem{lemma*}{Lemma}
\theoremstyle{plain}
\newtheorem{theorem}{Theorem}
\newtheorem{lemma}[theorem]{Lemma}
\newtheorem{corollary}[theorem]{Corollary}

\theoremstyle{definition}

\newtheorem{remark}{Remark}


\begin{document}

\title
[{Sum of elements in finite Sidon sets II}] {Sum of elements in finite Sidon sets II}

\author
[Yuchen Ding] {Yuchen Ding}

\address{(Yuchen Ding) School of Mathematical Science,  Yangzhou University, Yangzhou 225002, People's Republic of China}
\email{ycding@yzu.edu.cn}


\begin{abstract} A set $S\subset\{1,2,...,n\}$ is called a Sidon set if all the sums $a+b~~(a,b\in S)$ are different. Let $S_n$ be the largest cardinality of the Sidon sets in $\{1,2,...,n\}$. In a former article, the author proved the following asymptotic formula
$$\sum_{a\in S,~|S|=S_n}a=\frac{1}{2}n^{3/2}+O(n^{111/80+\varepsilon}),$$
where $\varepsilon>0$ is an arbitrary small constant.
In this note, we give an extension of the above formula. 
We show that
$$\sum_{a\in S,~|S|=S_n}a^{\ell}=\frac{1}{\ell+1}n^{\ell+1/2}+O\left(n^{\ell+61/160}\right)$$
for any positive integers $\ell$. 
Besides, we also consider the asymptotic formulae of other type summations involving Sidon sets. The proofs are established in a more general setting, namely we obtain the asymptotic formulae of the Sidon sets with $t$ elements when $t$ is near the magnitude $n^{1/2}$.
\end{abstract}
\maketitle

\baselineskip 18pt

\section{Introduction}
The notion of the Sidon set dates back to a paper of Sidon \cite{Si} where he investigated the coefficients of the Fourier series. Since then, the researches of Sidon sets became rich after the past ninety years. Let $n$ be a positive integer and $S$ be a subset of the first $n$ numbers $\{1,2,...,n\}$. We call $S$ a Sidon set if all the sums $a+b~~(a,b\in S)$ are different. Denoting by $S_n$ the largest cardinality of the Sidon sets in $\{1,2,...,n\}$, i.e.,
$$S_n=\max_{S\subset\{1,2,...,n\},~S~\text{Sidon}}|S|.$$

It is a natural question to ask the magnitude of the number $S_n$. And there are a few literatures involving this magnitude. In 1941, Erd\H{o}s and Tur\'{a}n \cite{ET} proved that
$$\frac{1}{\sqrt{2}}\leqslant\liminf_{n\rightarrow\infty}\frac{S_n}{\sqrt{n}}\leqslant\limsup_{n\rightarrow\infty}\frac{S_n}{\sqrt{n}}\leqslant1.$$
Actually, they gave more explicit upper bound $S_n\leqslant n^{1/2}+O(n^{1/4})$. This upper bound was sharpen to $S_n\leqslant n^{1/2}+n^{1/4}+1$  by Lindstr\"{o}m \cite{Li}. Later, Cilleruelo \cite{Ci} improved slightly this bound by showing that $S_n\leqslant n^{1/2}+n^{1/4}+1/2$. In about 2017, Professor Yong--Gao Chen in Nanjing Normal University showed me that the original method of Erd\H{o}s and Tur\'{a}n with some more elaborate calculations can already lead to the bound $S_n\leqslant n^{1/2}+n^{1/4}+5/8$ (private communications). Recently, Balogh, F\"{u}redi and Roy \cite{BFR} proved that $S_n\leqslant n^{1/2}+0.998n^{1/4}$ for sufficiently large $n$ by combining the methods of Erd\H{o}s--Tur\'{a}n and Lindstr\"{o}m.

In the other direction, Chowla \cite{Ch} and Erd\H{o}s \cite{Er} independently found that the constructions of Singer \cite{Sin} and Bose \cite{Bo} provided the evidences of $S_n\geqslant (1+o(1))n^{1/2}$. Based on the results of gaps between primes (up to the year of 1999), it has been remarked by Ruzsa \cite{Ru} that $S_n\geqslant n^{1/2}+O(n^{11/40+\varepsilon})$, where $\varepsilon>0$ is an arbitrary small constant. This lower bound together with the upper bound give us that
\begin{equation}\label{eq1}
S_n=n^{1/2}+O(n^{11/40+\varepsilon}).
\end{equation}
It was conjectured by Erd\H os \cite{Er2} that for any $\varepsilon>0$, we have $S_n=n^{1/2}+O(n^{\varepsilon})$. He offered one thousand dollars for the first one to solve this conjecture.

In the famous book `{\it Unsolved Problems in Number Theory}' of Guy \cite[Problem E28]{Gu}, it asked what is the superemum of the sum of reciprocal terms taken from a Sidon set. It has been proved that
$$\max_{\substack{S\subseteq \{1,2,..,n\}\\\text{Sidon}}}\sum_{a\in S}\frac{1}{a}$$ is near $2$ for sufficiently large $n$. For results of this topic, one can refer to \cite{Abb,Erd,Lev,LO',MC,O'S,TY,Zha1,Zha2,Zha3}.  

For the collections of researches about Sidon sets, one can refer to the excellent survey article \cite{O'} and P.hD. thesis \cite{O'2} of O’Bryant.

In the Lindstr\"om argument for bounding the
maximal size of a Sidon set, it arises naturally the estimate of 
$$\sum_{\substack{S=\{ a_1<a_2<\cdot\cdot\cdot <a_t\}\\S\subset[1,n]~\text{Sidon}}}(n+1-i)a_i.$$
In 2021, the author \cite{Di} considered the magnitude of the sum of elements in Sidon sets. To be precise, let $S$ be a Sidon set in $\{1,2,...,n\}$ with $|S|=S_n$, using the asymptotic formula of $S_n$ in equation (\ref{eq1}) the author proved for any $\varepsilon>0$
\begin{equation}\label{eq2}
\sum_{a\in S}a=\frac{1}{2}n^{3/2}+O(n^{111/80+\varepsilon}).
\end{equation}
In this subsequent note, we provide the asymptotic formulae of the following two type summations:
$$\sum_{\substack{S=\{ a_1<a_2<\cdot\cdot\cdot <a_t\}\\S\subset[1,n]~\text{Sidon}}}a_i^\ell~~~\qquad~~~\text{and}~~~\sum_{\substack{S=\{a_1<a_2<\cdot\cdot\cdot <a_t\}\\S\subset[1,n]~\text{Sidon}}}ia_i^\ell.$$
Using these estimates with $\ell=1$, we finally attain the asymptotic formula of 
$$\sum_{\substack{S=\{a_1<a_2<\cdot\cdot\cdot <a_t\}\\S\subset[1,n]~\text{Sidon}}}(n+1-i)a_i$$
which arose in the text of Lindstr\"om. Although this formula can not effect the upper bound of the maximal Sidon sets in Lindstr\"om's argument, this is of interest itself.

Our first theorem deals with the general case of equation (\ref{eq2}).

\begin{theorem}\label{th1}Let $S=\{ a_1<a_2<\cdot\cdot\cdot <a_t\}$ be a Sidon set in $\{1,2,...,n\}$. Then for any positive integer $\ell$, we have
$$\sum_{i=1}^ta_i^\ell=
\begin{cases}
\frac{1}{\ell+1}tn^{\ell}+O_\ell\left(n^{\ell+3/8}\right),~&~\text{if~}t\geqslant  n^{1/2}-n^{1/4}\\
\\
\frac{1}{\ell+1}tn^{\ell}+O_\ell\left(n^{\ell+1/4}\sqrt{n^{1/2}-t}\right),~&\text{if~}t<n^{1/2}-n^{1/4}.
\end{cases}$$
\end{theorem}
Several corollaries can be deduced from Theorem \ref{th1}. The asymptotic formula mentioned in the abstract is proved by taking $t=S_n=n^{1/2}+O(n^{21/80})$ (see Lemma \ref{lem3}).

Taking $\ell=1$, we have the following corollary, which is another extension of equation (\ref{eq2}).

\begin{corollary}\label{co1}Let $S=\{ a_1<a_2<\cdot\cdot\cdot <a_t\}$ be a Sidon set in $\{1,2,...,n\}$. Then
$$\sum_{i=1}^ta_i=
\begin{cases}
\frac{1}{2}nt+O(n^{11/8}),~&\text{if~} t\geqslant n^{1/2}-n^{1/4}\\
\\
\frac{1}{2}nt+O\left(n^{5/4}\sqrt{n^{1/2}-t}\right),~&\text{if~}t<n^{1/2}-n^{1/4}.
\end{cases}$$
\end{corollary}

Taking $t=S_n$ and then Corollary \ref{co1} reduces to the following Corollary \ref{co3}. 
\begin{corollary}\label{co3}Let $S$ be a Sidon set in $\{1,2,...,n\}$ with $|S|=S_n$, then we have
$$\sum_{a\in S}a=\frac{1}{2}n^{3/2}+O(n^{221/160}).$$
\end{corollary}
It is certainly that $11/8<221/160<111/80$, so this is a small improvement of the error term in equation (\ref{eq2}). In the former article, the author pointed out that one can improve the error term in equation (\ref{eq2}) to $O(n^{11/8})$ under the assumption $S_n=n^{1/2}+O(n^{1/4})$. This assumption is still open at present due to the barrier of the gaps between primes. By some applications of the mean value estimate rather than the direct use of the size of Sidon sets, we can prove the following stronger result comparing with Corollary \ref{co3}.

\begin{theorem}\label{th2}Let $S$ be a Sidon set in $\{1,2,...,n\}$ with $|S|=S_n$, then for all $n\leqslant N$ but at most $O(N/(\log N)^{7/19})$ exceptions, we have
$$\sum_{a\in S}a=\frac{1}{2}n^{3/2}+O(n^{11/8}\log n).$$
\end{theorem}
From Theorem \ref{th2}, we immediately have the following corollary.

\begin{corollary}\label{co4}
Let $S$ be a Sidon set in $\{1,2,...,n\}$ with $|S|=S_n$, then the asymptotic formula $\sum_{a\in S}a=\frac{1}{2}n^{3/2}+O(n^{11/8}\log n)$ is true for almost all integers $n$.
\end{corollary}

As an application of Theorem \ref{th1}, we establish the last theorem in this note.
\begin{theorem}\label{th3}Let $S=\{ a_1<a_2<\cdot\cdot\cdot <a_t\}$ be a Sidon set in $\{1,2,...,n\}$. Then for any positive integer $\ell$ we have
$$\sum_{i=1}^tia_i^\ell=\begin{cases}
\frac{1}{\ell+2}t^2n^{\ell}+O_\ell\left(n^{\ell+3/8}t\right),~&~\text{if~}t\geqslant  n^{1/2}-n^{1/4}\\
\\
\frac{1}{\ell+2}t^2n^{\ell}+O_\ell\left(n^{\ell+1/4}t\sqrt{n^{1/2}-t}\right),~&\text{if~}t<n^{1/2}-n^{1/4}.
\end{cases}$$
\end{theorem}

By Theorem \ref{th3}, we finally offer the asymptotic formula of the summation occurred in the Lindstr\"om article.

\begin{corollary}\label{coro4}Let $S=\{ a_1<a_2<\cdot\cdot\cdot <a_t\}$ be a Sidon set in $\{1,2,...,n\}$. Then
$$\sum_{i=1}^t(n+1-i)a_i=\begin{cases}
\frac{1}{2}n^2t+O\left(n^{19/8}\right),~&~\text{if~}t\geqslant  n^{1/2}-n^{1/4}\\
\\
\frac{1}{2}n^2t+O\left(n^{9/4}\sqrt{n^{1/2}-t}\right),~&\text{if~}t<n^{1/2}-n^{1/4}.
\end{cases}$$
\end{corollary}

\section{Lemmas}
In this section, we provide some lemmas which shall be used in the proofs of our theorems.
\begin{lemma}[Baker--Harman--Pintz]\cite{BHP}\label{lem1} Let $p_k$ be the $k$--th prime, then we have
$$p_{k+1}-p_k\ll p_k^{0.525}.$$
\end{lemma}

\begin{lemma}[Bose]\cite{Bo}\label{lem2} Let $p$ be a prime, then there are at least $p$ elements between $[1,p^2-1]$ such that all the sums of two of these elements are different modulo $p^2-1$. In other words, we have $S_{p^2-1}\geqslant p$ for primes $p$.
\end{lemma}

\begin{lemma}\label{lem3} For $n\geqslant 2$, we have $S_n=n^{1/2}+O(n^{21/80})$.
\end{lemma}
This is mentioned in \cite{O'} without proof and we offer one here for complement.
\begin{proof} Let $p_i$ be the $i$--th prime. Suppose that $p_t$ is the largest prime such that $p_t^2-1<n$, then we have
\begin{equation}\label{eq2.1}
p_t^2-1<n\leqslant p_{t+1}^2-1.
\end{equation}
By Lemmas \ref{lem1} and \ref{lem2}, it follows that
$$S_n\geqslant S_{p_{t}^2-1}\geqslant p_t=p_{t+1}-(p_{t+1}-p_t)\geqslant\sqrt{n+1}+O(p_t^{0.525}).$$
From the equation (\ref{eq2.1}), it is clear that $p_t^{0.525}\ll n^{0.525/2}=n^{21/80}$. Now the lemma follows from the fact that $S_n\leqslant n^{1/2}+O(n^{1/4})$.
\end{proof}

\begin{lemma}[Cilleruelo]\cite{Ci2}\label{lem4} Given a Sidon set $A\subset\{1,2,...,n\}$ with $|A|=n^{1/2}-L$. Then every subinterval $I\subset[1,n]$ with length $cn$ contains $c|A|+E_I$ elements of $A$, where
$$|E_I|\leqslant52n^{1/4}(1+c^{1/2}n^{1/8})(1+L_{+}^{1/2}n^{-1/8}),~~L_{+}=\max\{0,L\}.$$
\end{lemma}
Let $B_n$ denote the $n$--th Bernoulli number, i.e., it is given by
$$\frac{x}{e^x-1}=\sum_{n=0}^{\infty}\frac{B_n}{n!}x^n.$$
It is well known that (see for example \cite{KKS})
$$B_0=1,~B_1=-\frac{1}{2},~B_2=\frac{1}{6},~B_4=-\frac{1}{30},~B_6=\frac{1}{42},~B_8=-\frac{1}{30},~B_{10}=\frac{5}{66},...$$
and $B_n=0$ for odd integers $n\geqslant 3$. The Bernoulli polynomial $B_n(x)~(n=0,1,2,3,...)$ is defined by 
\begin{align}\label{52-3}
B_n(x)=\sum_{i=0}^{n}\binom niB_ix^{n-i},
\end{align}
where $\binom ni=\frac{n!}{i!(n-i)!}$. We have the following identity.
\begin{lemma}\cite[Page 93]{KKS}\label{lem5} 
For natural numbers $r$ and $x$, we have
$$\sum_{n=0}^{x-1}n^{r-1}=\frac{1}{r}(B_r(x)-B_r).$$
\end{lemma}

\section{Proofs}

\begin{proof}[ Proof of Theorem \ref{th1}] 
We start with calculations of the summation $\sum_{k=1}^{n}k^{\ell-1}S(k)$, where $S(k)=|S\cap [1,k]|$ for any $1\leqslant k\leqslant n$. On one hand,
\begin{align}\label{e-51-1}\sum_{k=1}^{n}k^{\ell-1}S(k)&=\sum_{k=1}^{n}k^{\ell-1}\sum_{\substack{a\in S\\1\leqslant a\leqslant k}}1=\sum_{a\in S}\sum_{a\leqslant k\leqslant n}k^{\ell-1}.
\end{align}
By Lemma \ref{lem5}, we have
\begin{align}\label{eq-521}
\sum_{a\leqslant k\leqslant n}k^{\ell-1}=\frac{1}{\ell}B_\ell(n+1)-\frac{1}{\ell}B_\ell(a).
\end{align}
Taking equation (\ref{eq-521}) into equation (\ref{e-51-1}) gives
\begin{align}
\sum_{k=1}^{n}k^{\ell-1}S(k)=\frac{1}{\ell}B_\ell(n+1)t-\frac{1}{\ell}\sum_{a\in S}B_\ell(a).
\end{align}
On the other hand, by Lemma \ref{lem4} with $A=S$ and $I=S\cap [1,k]$ and Lemma \ref{lem5}, 
\begin{align}\label{e-51-2}\sum_{k=1}^{n}k^{\ell-1}S(k)&=\sum_{k=1}^{n}k^{\ell-1}|S\cap [1,k]|=\sum_{k=1}^{n}k^{\ell-1}\left(\frac{k}{n}t+E_k\right)\nonumber\\
&=\frac{t}{n(\ell+1)}\left(B_{\ell+1}(n+1)-B_{\ell+1}\right)+\sum_{k=1}^{n}k^{\ell-1}E_k,
\end{align}
where $|E_k|\leqslant 52n^{1/4}(1+(k/n)^{1/2}n^{1/8})(1+L_{+}^{1/2}n^{-1/8}),~~L_{+}=\max\{0,n^{1/2}-t\}.$ Hence, we get
\begin{align}\label{52-2}
\sum_{a\in S}B_\ell(a)&=\left(B_\ell(n+1)-\frac{\ell}{\ell+1}\frac{B_{\ell+1}(n+1)}{n}\right)t+\frac{\ell t}{(\ell+1)n}B_{\ell+1}-\sum_{k=1}^{n}k^{\ell-1}E_k.
\end{align}
From the definition of the Bernoulli polynomial (equation (\ref{52-3})), we have
\begin{align}\label{52-4}
B_\ell(n+1)-\frac{\ell}{\ell+1}\frac{B_{\ell+1}(n+1)}{n}=\frac{1}{\ell+1}n^{\ell}+O_\ell\left(n^{\ell-1}\right)
\end{align}
and
\begin{align}\label{52-5}
B_\ell(a)=a^\ell+O_\ell\left(a^{\ell-1}\right).
\end{align}
Therefore,
\begin{align}\label{52-6}
\sum_{i=1}^t a_i^\ell=\sum_{a\in S}a^\ell=\frac{1}{\ell+1}tn^{\ell}+O_\ell\left(tn^{\ell-1}\right)-\sum_{k=1}^{n}k^{\ell-1}E_k.
\end{align}
It remains to bound the sum involving $E_k$. 
Since
\begin{align}\label{52-7}
|E_k|&\ll 
\begin{cases}
n^{1/4}\left(1+(k/n)^{1/2}n^{1/8}\right),~&\text{if~} t\geqslant  n^{1/2}-n^{1/4}\\
n^{1/8}\left(1+(k/n)^{1/2}n^{1/8}\right)\sqrt{n^{1/2}-t},~&\text{if~}t<n^{1/2}-n^{1/4}
\end{cases}
\nonumber\\&\ll 
\begin{cases}
n^{1/4}+k^{1/2}n^{-1/8},~&\text{if~} t\geqslant n^{1/2}-n^{1/4}\\
\left(n^{1/8}+k^{1/2}n^{-1/4}\right)\sqrt{n^{1/2}-t},~&\text{if~}t<n^{1/2}-n^{1/4},
\end{cases}
\end{align}
it follows that
\begin{align}\label{52-8}
\sum_{k=1}^{n}k^{\ell-1}E_k\ll
\begin{cases}
n^{\ell+3/8},~&~\text{if~}t\geqslant  n^{1/2}-n^{1/4}\\
n^{\ell+1/4}\sqrt{n^{1/2}-t},~&\text{if~}t<n^{1/2}-n^{1/4}.
\end{cases}
\end{align}
Now the theorem follows from the trivial estimates
$$tn^{\ell-1}\ll n^{\ell+1/8}$$
for $t<n^{1/2}+O(n^{1/4})$ and 
$$tn^{\ell-1}\ll n^{\ell+1/4}\sqrt{n^{1/2}-t}$$
provided that $t<n^{1/2}-n^{1/4}$.
\end{proof}

\begin{proof}[ Proof of Theorem \ref{th2}] The proof is an application of the mean value idea. For any $1\leqslant n\leqslant N$, let $\mathscr{S}_n$ be a Sidon set in $\{1,2,...,n\}$ with $|\mathscr{S}_n|=S_n$. By Lemma \ref{lem3} and equation (\ref{52-2}) with $\ell=1$ and $S=\mathscr{S}_n$, we know that
\begin{align}\label{eq3.6}\sum_{a\in \mathscr{S}_n}a&=\left(B_1(n+1)-\frac{1}{2}\frac{B_2(n+1)}{n}\right)S_n+O\left(\frac{1}{\sqrt{n}}\right)-\sum_{k=1}^nE_{n,k}
\nonumber\\
&=\frac{1}{2}n^{3/2}+O(n^{101/80})-\sum_{k=1}^nE_{n,k},
\end{align}
where
\begin{align}\label{eq3.8}
|E_{n,k}|\leqslant 52n^{1/4}(1+(k/n)^{1/2}n^{1/8})(1+L_{n}^{1/2}n^{-1/8}),~~L_{n}=\max\{0,n^{1/2}-S_n\}.
\end{align}
In another word, we shall have
\begin{align}\label{eq3.7} \left|\sum_{a\in \mathscr{S}_n}a-\frac{1}{2}n^{3/2}\right|\ll n^{101/80}+\sum_{k=1}^n|E_{n,k}|.
\end{align}
By equation (\ref{eq3.7}), we immediately obtain that
\begin{align}\label{eq3.9}\sum_{n=1}^{N}\left|\sum_{a\in \mathscr{S}_n}a-\frac{1}{2}n^{3/2}\right|\ll N^{181/80}+\sum_{n=1}^{N}\sum_{k=1}^n|E_{n,k}|.
\end{align}
We now treat the summation involving $E_{n,k}$. From equation (\ref{eq3.8}), we know that
\begin{align}\label{eq3.11}\sum_{k=1}^{n}|E_{n,k}|\ll (n^{1/4}+L_n^{1/2}n^{1/8})\sum_{k=1}^{n}(1+k^{1/2}n^{-3/8})\ll n^{11/8}+L_n^{1/2}n^{5/4}.
\end{align}
From which it can be deduced that
\begin{align}\label{eq3.12}\sum_{n=1}^{N}\sum_{k=1}^{n}|E_{n,k}|\ll\sum_{n=1}^{N} (n^{11/8}+L_n^{1/2}n^{5/4})\ll N^{19/8}+\sum_{n=1}^{N}L_n^{1/2}n^{5/4}.
\end{align}
Suppose that $p_{\ell}$ is the least prime such that $N\leqslant p_{\ell}^2-1$, where $p_i$ is the $i$--th prime and we define $p_0$ to be $1$. Then we have
\begin{align}\label{eq3.13}\sum_{n=1}^{N}L_n^{1/2}n^{5/4}&\leqslant\sum_{m=0}^{\ell-1}\sum_{\substack{p_{m}^2-1<n\leqslant p_{m+1}^2-1}}L_n^{1/2}n^{5/4}.
\end{align}
Recall that $S_{p^2-1}\geqslant p$ from Lemma \ref{lem2}, thus by the definition of $L_n$ in equation (\ref{eq3.8}), we get
\begin{align}\label{eq3.14}\sum_{\substack{p_{m}^2-1<n\leqslant p_{m+1}^2-1}}L_n^{1/2}n^{5/4}&\leqslant\sum_{\substack{p_{m}^2-1<n\leqslant p_{m+1}^2-1}}n^{5/4}(n^{1/2}-p_m)^{1/2}\nonumber\\
&<\sum_{\substack{p_{m}^2<n\leqslant p_{m+1}^2}}n^{5/4}(n^{1/2}-p_m)^{1/2}\nonumber\\
&\ll p_{m+1}^{5/4}\sum_{\substack{p_{m}^2<n\leqslant p_{m+1}^2}}(n^{1/2}-p_m)^{1/2}\nonumber\\
&\ll p_{m+1}^{5/4}\int_{p_m^2}^{p_{m+1}^2}(x^{1/2}-p_m)^{1/2}dx\nonumber\\
&\xlongequal[x=(t+p_m)^2]{t=x^{1/2}-p_m}p_{m+1}^{5/4}\int_{0}^{p_{m+1}-p_m}
2t^{1/2}(t+p_m)dt\nonumber\\
&\ll p_{m+1}^{5/4}(p_{m+1}-p_m)^{5/2}+p_{m+1}^{5/4}p_m(p_{m+1}-p_m)^{3/2}\nonumber\\
&\ll p_{m+1}^{9/4}(p_{m+1}-p_m)^{3/2}.
\end{align}
Taking this into equation (\ref{eq3.13}), we get
\begin{align}\label{eq3.15}\sum_{n=1}^{N}L_n^{1/2}n^{5/4}&\ll\sum_{m=0}^{\ell-1}p_{m+1}^{9/4}(p_{m+1}-p_m)^{3/2}\nonumber\\
&\leqslant\sum_{m=0}^{\ell-1}p_{m+1}^{11/4}(p_{m+1}-p_m)\nonumber\\
&\leqslant p_{\ell}^{11/4}\sum_{m=0}^{\ell-1}(p_{m+1}-p_m)\nonumber\\
&\leqslant p_{\ell}^{15/4}.
\end{align}
Recall that $p_{\ell-1}^2-1<N\leqslant p_{\ell}^2-1$ from the definition of $\ell$, so we have
\begin{align}\label{equation} p_{\ell}=p_{\ell-1}+(p_{\ell}-p_{\ell-1})\leqslant\sqrt{N+1}+O(p_{\ell-1}^{0.525})\ll N^{1/2}.
\end{align}
It follows that 
\begin{align}\label{equation2} \sum_{n=1}^{N}L_n^{1/2}n^{5/4}\ll N^{15/8}.
\end{align}
Thus, we conclude that
\begin{align}\label{eq3.16} \sum_{n=1}^{N}\left|\sum_{a\in \mathscr{S}_n}a-\frac{1}{2}n^{3/2}\right|\ll N^{19/8}
\end{align}
from equations (\ref{eq3.9}), (\ref{eq3.12}) and (\ref{equation2}). Now let $\mathscr{A}$ be the set of integers $n\leqslant N$ such that
\begin{align}\label{eq3.17} \left|\sum_{\substack{a\in \mathscr{S}_n}}a-\frac{1}{2}n^{3/2}\right|>n^{11/8}\log n.
\end{align}
If there are more than $N/(\log N)^{7/19}$ integers in $\mathscr{A}$, we shall have
\begin{align}\label{eq3.18} \sum_{n\in\mathscr{A}}\left|\sum_{\substack{a\in \mathscr{S}_n}}a-\frac{1}{2}n^{3/2}\right|>\sum_{n\leqslant N/(\log N)^{7/19}}n^{11/8}\log n\gg N^{19/8}(\log N)^{1/8}.
\end{align}
This a contradiction with equation (\ref{eq3.16}) for sufficiently large $N$, which means that we should have $|\mathscr{A}|<N/(\log N)^{7/19}$ provided that $N$ is large enough. Therefore, for all integers $n\leqslant N$ but at most $N/(\log N)^{7/19}$ exceptions, we have
\begin{align}\label{equation3} \left|\sum_{\substack{a\in \mathscr{S}_n}}a-\frac{1}{2}n^{3/2}\right|\leqslant n^{11/8}\log n.
\end{align}
\end{proof}

\begin{remark} It is sure that the error term $O(n^{11/8}\log n)$ in Theorem \ref{th2} can be improved to $O(n^{11/8}f(n))$ for the integers of density one provided that $f(n)$ tends to infinity. It seems that this is the best possible which can be achieved by the present method.
\end{remark}

\begin{proof}[ Proof of Theorem \ref{th3}]
It is clear that
\begin{align}\label{t3-1}
\sum_{i=1}^{t}ia_i^{\ell}=\sum_{a\in S}a^{\ell}\sum_{\substack{b\in S\\b\leqslant a}}1=\sum_{a\in S}a^\ell\left(\frac{a}{n}t+E_a\right)=\frac{t}{n}\sum_{a\in S}a^{\ell+1}+\sum_{a\in S}a^\ell E_a
\end{align}
from Lemma \ref{lem4} with $A=S$ and $I=S\cap[1,a]$, where 
\begin{align}\label{t3-2}
|E_a|&\leqslant 52n^{1/4}(1+(a/n)^{1/2}n^{1/8})(1+L_{+}^{1/2}n^{-1/8})~~(\text{with}~~L_{+}=\max\{0,n^{1/2}-t\})\nonumber\\
&\ll 
\begin{cases}
n^{1/4}+a^{1/2}n^{-1/8},~&\text{if~} t\geqslant n^{1/2}-n^{1/4}\\
\left(n^{1/8}+a^{1/2}n^{-1/4}\right)\sqrt{n^{1/2}-t},~&\text{if~}t<n^{1/2}-n^{1/4}.
\end{cases}
\end{align}
For $t\geqslant n^{1/2}-n^{1/4}$, we have
\begin{align}\label{t3-3}
\sum_{a\in S}a^\ell E_a&\ll n^{1/4}\sum_{a\in S}a^\ell+n^{-1/8}\sum_{a\in S}a^{\ell+1/2}\nonumber\\
&\ll n^{1/4}n^\ell t+n^{-1/8}n^{\ell+1/2}t\nonumber\\
&\ll n^{\ell+3/8}t.
\end{align}
For $t<n^{1/2}-n^{1/4}$, we have
\begin{align}\label{t3-4}
\sum_{a\in S}a^\ell E_a&\ll n^{1/8}\sqrt{n^{1/2}-t}\sum_{a\in S}a^\ell+n^{-1/4}\sqrt{n^{1/2}-t}\sum_{a\in S}a^{\ell+1/2}\nonumber\\
&\ll n^{1/8}n^\ell t\sqrt{n^{1/2}-t}+n^{-1/4}n^{\ell+1/2}t\sqrt{n^{1/2}-t}\nonumber\\
&\ll n^{\ell+1/4}t\sqrt{n^{1/2}-t}.
\end{align}
Recall that 
\begin{align}\label{t3-5}
\sum_{a\in S}a^{\ell+1}=\begin{cases}
\frac{1}{\ell+2}tn^{\ell+1}+O_\ell\left(n^{\ell+11/8}\right),~&~\text{if~}t\geqslant  n^{1/2}-n^{1/4}\\
\frac{1}{\ell+2}tn^{\ell+1}+O_\ell\left(n^{\ell+5/4}\sqrt{n^{1/2}-t}\right),~&\text{if~}t<n^{1/2}-n^{1/4}.
\end{cases}
\end{align}
from Theorem \ref{th1},
thus the theorem follows from equations (\ref{t3-1}), (\ref{t3-3}), (\ref{t3-4}) and (\ref{t3-5}).
\end{proof}

\begin{proof}[Proof of Corollary \ref{coro4}] From Corollary \ref{co1} and Theorem \ref{th3} with $\ell=1$, we have
\begin{align}\label{eqq}
\sum_{i=1}^t(n+1-i)a_i&=(n+1)\sum_{i=1}^ta_i-\sum_{i=1}^tia_i\nonumber\\
&=\begin{cases}
\frac{1}{2}n^2t+O\left(n^{19/8}\right),~&~\text{if~}t\geqslant  n^{1/2}-n^{1/4}\\
\frac{1}{2}n^2t+O\left(n^{9/4}\sqrt{n^{1/2}-t}\right),~&\text{if~}t<n^{1/2}-n^{1/4}.
\end{cases}
\end{align}
\end{proof}

\section*{Acknowledgments}
The author is
supported by the Natural Science Foundation of Jiangsu Province of
China, Grant No. BK20210784. He is also supported by the
foundations of the projects ``Jiangsu Provincial
Double--Innovation Doctor Program'', Grant No. JSSCBS20211023 and
``Golden  Phoenix of the Green City--Yang Zhou'' to excellent PhD,
Grant No. YZLYJF2020PHD051.

\end{document}